\documentclass[a4paper,12pt]{article}
\usepackage{amsmath,amsthm,amssymb}
\usepackage[T2A]{fontenc}
\usepackage{amscd}

\usepackage{MnSymbol} 

\usepackage[table]{xcolor} 

\usepackage{hyperref}

  \newskip\prethm \prethm3.0pt plus1.3pt minus.4pt
  \newskip\posthm \posthm2.7pt plus1.4pt minus.3pt
  \newtheoremstyle{STATEMENT}%
       {\prethm}{\posthm}{\itshape}{\parindent}{\scshape}
       {.}{.6em plus.2em minus.1em}{}
  \newtheoremstyle{EXPLANATION}%
       {\prethm}{\posthm}{}{\parindent}{\scshape}
       {.}{.6em plus.2em minus.1em}{}
\theoremstyle{STATEMENT}

\newtheorem{theorem}{Theorem}[section]

\newtheorem{lemma}{Lemma}[section]
\newtheorem{assertion}{Assertion}[section]
\newtheorem{corollary}{Corollary}[section]

\theoremstyle{EXPLANATION}
\newtheorem{definition}{Definition}[section]
\newtheorem{remark}{Remark}[section]

\newtheorem{example}{Example}[section]

\begin{document}

\title{Local bi-integrability of bi-Hamiltonian systems \\ via bi-Poisson reduction}
\author{I.\,K.~Kozlov\thanks{No Affiliation, Moscow, Russia. E-mail: {\tt ikozlov90@gmail.com} }
}
\date{}

\maketitle

\begin{abstract} We prove that any bi-Hamiltonian system $v = \left(\mathcal{A} + \lambda \mathcal{B}\right)dH_{\lambda}$ that is Hamiltonian with respect all Poisson brackets $\mathcal{A} + \lambda \mathcal{B}$ is locally bi-integrable in both the real smooth case, when all eigenvalues of the Poisson pencil  $\mathcal{P} = \left\{\mathcal{A} + \lambda \mathcal{B}\right\}$ are real, and in the complex analytic case. A complete set of functions in bi-involution is constructed by extending the set of standard integrals, which consists of Casimir functions of Poisson brackets, eigenvalues of the Poisson pencil and Hamiltonians. \end{abstract}

\tableofcontents

\section{Introduction}

Let $M$ be a real $C^{\infty}$-smooth or complex analytic manifold. Two Poisson brackets $\mathcal{A}$ and $\mathcal{B}$ on $M$ are  \textbf{compatible}, if any their linear combination with constant coefficients is also a Poisson bracket. The \textbf{Poisson pencil} generated by these 
compatible Poisson brackets is the set \[ \mathcal{P} = \left\{ \mathcal{A}_{\lambda} = \mathcal{A} + \lambda \mathcal{B} \right\}_{\lambda \in \bar{\mathbb{K}}},\] where $\mathbb{K}=\mathbb{R}$ or $\mathbb{C}$ is the underlying field, $\bar{\mathbb{K}} = \mathbb{K} \cup \left\{\infty\right\}$ and $\mathcal{A}_{\infty} = \mathcal{B}$.

\begin{definition}
A dynamical system $v = \dot{x}$ is called \textbf{bi-Hamiltonian w.r.t. a pencil} $\mathcal{P}$ if it is Hamiltonian w.r.t. all brackets of the pencil, i.e. for any $\lambda \in \bar{\mathbb{K}}$ there exists a smooth function $H_{\lambda}$ such that \begin{equation} \label{Eq:BiHam1} v = \mathcal{A}_{\lambda} d H_{\lambda}. \end{equation} \end{definition}

Since the pioneering work by Franco Magri \cite{Magri78} (which was futher developed in \cite{Gelfand79}, \cite{Magri84} and \cite{Reiman80}), it is well known that integrability of many systems in mathematical physics, geometry and mechanics is closely related to their bi-Hamiltonian nature. In this paper we consider bi-Hamiltonian systems on a finite-dimensional 
 manifold $M^n$. There are two well-known ways of constructing integrals of motion for such systems:

\begin{enumerate}

\item If a system is Hamiltonian w.r.t. a Poisson pencil $\mathcal{P}$, then the \textit{Casimir functions of each bracket} $\mathcal{A}_{\lambda}$ are first integrals of the system. This is precisely the way how the argument shift method by A.\,S.~Mishchenko \& A.\,T.~Fomenko \cite{ArgShift} constructs integrals.

\item If the compatible Poisson brackets $\mathcal{A}$ and $\mathcal{B}$ are nondegenerate, then the \textit{eigenvalues} $\lambda_j(x)$ of the recursion operator $P = \mathcal{A}\mathcal{B}^{-1}$ are integrals of motion a bi-Hamiltonian system \eqref{Eq:BiHam1}. In a general case, the recursion operator $P = \mathcal{A}\mathcal{B}^{-1}$ is not defined, but we can still construct eigenvalues $\lambda_j(x)$ using the Kronecker Canonical Form (KCF) for pairs of skew-symmetric matrices given by the Jordan--Kronecker theorem (see e.g. \cite{BolsZhang}).

\end{enumerate}

It is well-known that, under some mild conditions of regularity, the Casimir functions and the eigenvalues $\lambda_j(x)$ are in \textbf{bi-involution}, i.e. they commute w.r.t. all brackets $\mathcal{A}_{\lambda} = \mathcal{A} + \lambda \mathcal{B}$ (this also follows from Lemma~\ref{L:StandIntBiInv}). The \textbf{rank} of a Poisson pencil $\mathcal{P} = \left\{\mathcal{A} + \lambda \mathcal{B}\right\}$ on $M$ is \begin{equation} \label{Eq:PencilRank} \operatorname{rk} \mathcal{P}=\max_{x \in M, \lambda \in \bar{\mathbb{K}}} \left( \mathcal{A}(x) + \lambda \mathcal{B}(x) \right). \end{equation}  In practice, we often get a \textbf{complete} set of functions, meaning  that we get \begin{equation} \label{Eq:BiIntNumber} N = \dim M - \frac{1}{2}  \operatorname{rk} \mathcal{P} \end{equation} functionally independent integrals $f_1, \dots, f_N$ (see e.g. \cite{BolsZhang}, \cite{Izosimov14},  \cite{Kozlov23Shift} and the references therein). Functional Independence means that $df_1 \wedge \dots \wedge df_N \not = 0$ almost everywhere.

\begin{definition} A bi-Hamiltonian system \eqref{Eq:BiHam1} on a manifold $M$ is \textbf{locally bi-integrable} if  in a neighborhood of a generic point $x \in M$ there exist a complete set of integrals in bi-involution.  \end{definition}

In this paper, we provide answers to the long-standing questions \cite[Problem 9]{BolsinovIzosimomKonyaevOshemkov12} and \cite[Problem 12]{BolsinovTsonev17}  which inquire about local bi-integrability of bi-Hamiltonian systems. Our main result is the following.

\begin{theorem} \label{T:MultHamBiInt} Let $\mathcal{P} = \left\{ \mathcal{A} + \lambda \mathcal{B}\right\}$ be a Poisson pencil on a  real $C^{\infty}$-smooth or complex analytic manifold $M$. In the real case, we assume that all eigenvalues of $\mathcal{P}$ are real. If a vector field $v$ is bi-Hamiltonian w.r.t. $\mathcal{P}$, then it is locally bi-integrable. \end{theorem}

We prove a more general Theorem~\ref{T:MainGenTh} in Section~\ref{S:MainTh}. This theorem specifies the points $x\in M$ where the integrals are functionally independent and describes the standard integrals utilized in the analysis.

The term ``bi-Hamiltonian'' is often used to refer to systems that are Hamiltonian w.r.t. two Poisson brackets, which can be expressed as \begin{equation} \label{Eq:BiHamOnly2} v = \mathcal{A} df = \mathcal{B} dg.\end{equation} The following statement reveals the primary cause  why not all systems~\eqref{Eq:BiHamOnly2} are locally bi-integrable.

\begin{lemma} \label{L:FirstMainT} Let $\mathcal{A}$ and $\mathcal{B}$ be compatible Poisson brackets on $M$. If a vector field  $v = \mathcal{A} df = \mathcal{B} dg$ is locally bi-integrable, that it is tangent to the symplectic leaves $\operatorname{Im} \mathcal{A}_{\lambda}$ for a generic pair $(\lambda, x) \in \bar{\mathbb{C}} \times M$: \begin{equation} \label{Eq:ImCond} v(x) \in \operatorname{Im} \mathcal{A}_{\lambda}(x).\end{equation} 
\end{lemma}

\begin{proof}[Proof of Lemma~\ref{L:FirstMainT}] Let $f_1, \dots, f_m$ be a complete family of integrals in bi-involution on $M$. By Assertion~\ref{A:BiLagr} the subspace \[ L_x = \operatorname{span} \left\{ df_1(x), \dots, df_m (x)\right\}\] is a bi-Lagrangian subspace in $T^*_x M$ for a generic $x \in M$. Then $v(x) \in L_x^{0}$ and by Assertion~\ref{A:DualCoreLagr} Condition~\eqref{Eq:ImCond} holds for generic $(\lambda, x) \in \bar{\mathbb{C}} \times M$. Lemma~\ref{L:FirstMainT} is proved. \end{proof}

\begin{example} Condition \eqref{Eq:ImCond} is not satisfied for the following bi-Hamiltonian system \[\mathcal{A} =  \left(
\begin{array}{c|c}
  0 & \begin{matrix}
   1 & 0      &             \\
      & 1 & 0     \\
    \end{matrix} \\
  \hline
  \begin{matrix}
  \minus1  &           \\
  0   & \minus1    \\
      &         0  \\
  \end{matrix} & 0
 \end{array}
 \right) \quad  B=\left(\begin{array}{c|c}
  0 & \begin{matrix}
   0 & 1      &             \\
      & 0 & 1     \\
    \end{matrix} \\
  \hline
  \begin{matrix}
   0 &           \\
  \minus1    &   0 \\
      &         \minus1   \\
  \end{matrix} & 0
 \end{array}
 \right), \quad v = \left(\begin{matrix} 0 \\ 0 \\ \hline 0 \\ 1 \\ 0  \end{matrix} \right) \] and hence it is not (locally) bi-integrable. \end{example}

\begin{remark} The results of this paper can be generalized:

\begin{enumerate}

\item Theorem~\ref{T:MultHamBiInt} remains true in the general real $C^{\infty}$-case. if the Poisson pencil $\mathcal{P}$ possesses complex-conjugate eigenvalues $\alpha_j \pm i \beta_j$, the proof requires modification. The key step is to perform bi-Poisson reduction w.r.t. the standard integrals. Subsequently, we can utilize the results obtained by F.-J. Turiel in~\cite{turiel} regarding flat Poisson pencils.

\item Lemma~\ref{L:FirstMainT} can be extended to serve as a criterion for local bi-integrability. A bi-Hamiltonian system   $v = \mathcal{A} df = \mathcal{B} dg$ is locally bi-integrable if only if \eqref{Eq:ImCond} holds for generic $(\lambda, x) \in \bar{\mathbb{C}} \times M$. The proofs can be adapted to accommodate scenarios where the system is Hamiltonian not for all Poisson brackets $\mathcal{A}_{\lambda}, \lambda \in \bar{\mathbb{K}}$ but only for the values $\lambda$ belonging to an open subset $U \subset \bar{\mathbb{K}}$. 

\end{enumerate}

To preserve the simplicity of the proof, we focused on the most basic case, where the eigenvalues of $\mathcal{P}$ belong to the underlying field and the system is Hamiltonian w.r.t. all brackets $\mathcal{A}_{\lambda}$.

\end{remark}

\subsection{Conventions and acknowledgements}

The results of this paper are more or less expected by the experts in the field. The author was told by A.\,M.Izosimov \cite{IzosimovPriv}  that in his correspondence with A.\,V.~Bolsinov they discussed the proof of local bi-integrability using similar technique. 

\textbf{Conventions.} All manifolds (functions, Poisson brackets, etc)  are  either real $C^{\infty}$-smooth or complex analytic. Some property holds ``almost everywhere'' or ``at a generic point'' of a manifold $M$ if it holds on an open dense subset of $M$.  We denote $\bar{\mathbb{K}} = \mathbb{K} \cup \left\{ \infty \right\}$, where $\mathbb{K} = \mathbb{R}$ or $\mathbb{C}$ is the underlying field. The annihilator of a vector subspace $U$ is denoted by $U^{0}$. Sometimes we abbreviate ``Jordan--Kronecker'' as JK and ``Kronecker Canonical Form'' as KCF. We refer to the splitting theorem of Alan Weinstein (see  e.g. \cite[Theorem 4.2]{SilvaWeinstein99} or \cite[Theorem 1.4.5]{DufourZung05}), as the Darboux--Weinstein theorem.

\par\medskip

\textbf{Acknowledgements.} The author would like to thank A.\,V.~Bolsinov, A.\,M.~Izosimov and A.\,Yu.~Konyaev for useful comments.

\section{Linear algebra}

In this section we present some basic facts related to the geometry of a finite-dimensional complex vector space $V$ equipped with a pair of skew-symmetric bilinear forms $A, B$. We call a one-parametric family of skew-symmetric forms \[\mathcal{L} = \left\{A + \lambda B \, \, \bigr| \,\, \lambda \in \bar{\mathbb{C}} \right\}\] a \textbf{linear pencil}. The statements presented in this section are rooted in elementary linear algebra, and therefore, we will omit the formal proofs.

\subsection{Jordan--Kronecker theorem and Kronecker canonical form} 

First, let us recall the \textbf{Kronecker Canonical Form} (KCF) for a pair of skew-symmetric forms. The next theorem that describes it, which we call the Jordan--Kronecker theorem, is a classical result that goes back to Weierstrass and Kronecker.  A proof of it can be found in \cite{Thompson}, which is based on \cite{Gantmaher88}.

\begin{theorem}[Jordan--Kronecker theorem]\label{T:Jordan-Kronecker_theorem}
Let $A$ and $B$ be skew-symmetric bilinear forms on a complex
finite-dimension vector space $V$. There exists a basis of the space $V$
such that the matrices of both forms $A$ and $B$ are block-diagonal
matrices:

{\footnotesize
$$
A =
\begin{pmatrix}
A_1 &     &        &      \\
    & A_2 &        &      \\
    &     & \ddots &      \\
    &     &        & A_k  \\
\end{pmatrix}
\quad  B=
\begin{pmatrix}
B_1 &     &        &      \\
    & B_2 &        &      \\
    &     & \ddots &      \\
    &     &        & B_k  \\
\end{pmatrix}
$$
}

where each pair of corresponding blocks $A_i$ and $B_i$ is one of
the following:

\begin{itemize}

\item Jordan block with eigenvalue $\lambda_i \in \mathbb{C}$: {\scriptsize  \begin{equation} \label{Eq:JordBlockL} A_i =\left(
\begin{array}{c|c}
  0 & \begin{matrix}
   \lambda_i &1&        & \\
      & \lambda_i & \ddots &     \\
      &        & \ddots & 1  \\
      &        &        & \lambda_i   \\
    \end{matrix} \\
  \hline
  \begin{matrix}
  \minus\lambda_i  &        &   & \\
  \minus1   & \minus\lambda_i &     &\\
      & \ddots & \ddots &  \\
      &        & \minus1   & \minus\lambda_i \\
  \end{matrix} & 0
 \end{array}
 \right)
\quad  B_i= \left(
\begin{array}{c|c}
  0 & \begin{matrix}
    1 & &        & \\
      & 1 &  &     \\
      &        & \ddots &   \\
      &        &        & 1   \\
    \end{matrix} \\
  \hline
  \begin{matrix}
  \minus1  &        &   & \\
     & \minus1 &     &\\
      &  & \ddots &  \\
      &        &    & \minus1 \\
  \end{matrix} & 0
 \end{array}
 \right)
\end{equation}} \item Jordan block with eigenvalue $\infty$ {\scriptsize \begin{equation} \label{Eq:JordBlockInf}
A_i = \left(
\begin{array}{c|c}
  0 & \begin{matrix}
   1 & &        & \\
      &1 &  &     \\
      &        & \ddots &   \\
      &        &        & 1   \\
    \end{matrix} \\
  \hline
  \begin{matrix}
  \minus1  &        &   & \\
     & \minus1 &     &\\
      &  & \ddots &  \\
      &        &    & \minus1 \\
  \end{matrix} & 0
 \end{array}
 \right)
\quad B_i = \left(
\begin{array}{c|c}
  0 & \begin{matrix}
    0 & 1&        & \\
      & 0 & \ddots &     \\
      &        & \ddots & 1  \\
      &        &        & 0   \\
    \end{matrix} \\
  \hline
  \begin{matrix}
  0  &        &   & \\
  \minus1   & 0 &     &\\
      & \ddots & \ddots &  \\
      &        & \minus1   & 0 \\
  \end{matrix} & 0
 \end{array}
 \right)
 \end{equation} } \item   Kronecker block {\scriptsize \begin{equation} \label{Eq:KronBlock} A_i = \left(
\begin{array}{c|c}
  0 & \begin{matrix}
   1 & 0      &        &     \\
      & \ddots & \ddots &     \\
      &        & 1    &  0  \\
    \end{matrix} \\
  \hline
  \begin{matrix}
  \minus1  &        &    \\
  0   & \ddots &    \\
      & \ddots & \minus1 \\
      &        & 0  \\
  \end{matrix} & 0
 \end{array}
 \right) \quad  B_i= \left(
\begin{array}{c|c}
  0 & \begin{matrix}
    0 & 1      &        &     \\
      & \ddots & \ddots &     \\
      &        &   0    & 1  \\
    \end{matrix} \\
  \hline
  \begin{matrix}
  0  &        &    \\
  \minus1   & \ddots &    \\
      & \ddots & 0 \\
      &        & \minus1  \\
  \end{matrix} & 0
 \end{array}
 \right)
 \end{equation} }
 \end{itemize}

\end{theorem}

Each Kronecker block is a $(2k_i-1) \times (2k_i-1)$ block, where
$k_i \in \mathbb{N}$. If $k_i=1$, then the blocks are $1\times 1$
zero matrices \[A_i =
\begin{pmatrix}
0
\end{pmatrix}, \qquad B_i=
\begin{pmatrix}
0
\end{pmatrix}.\]  We call a decomposition of $V$ into a sum of subspaces corresponding to the Jordan and Kronecker blocks a \textbf{Jordan-Kronecker decomposition}:  \begin{equation} \label{Eq:JKDecomp} V= \bigoplus_{j=1}^{N_J}  V_{J_{\lambda_j, 2n_j}}\oplus  \bigoplus_{i=1}^{N_K} V_{K_i}.\end{equation}

\subsection{Core and mantle subspaces}

The \textbf{rank} of a linear pencil $\mathcal{L} = \left\{ A + \lambda B\right\}$ is \[ \operatorname{rk} \mathcal{L} = \max_{\lambda \in\bar{\mathbb{C}}} \operatorname{rk} (A +\lambda B).\] A value $\lambda_0 \in \bar{\mathbb{C}}$ is \textbf{regular} if $\operatorname{rk} A_{\lambda_0} = \operatorname{rk} \mathcal{L}$. We also call $A_{\lambda_0}$ a regular form of the pencil $\mathcal{L}$. Non-regular values $\lambda_0$ are \textbf{singular}. We denote the set of eigenvalues as \[\Lambda(\mathcal{L}) = \left\{ \lambda_j \,\, \bigr| \,\, \operatorname{rk}(A - \lambda_j B) < \operatorname{rk} \mathcal{L}\right\}.\] Due to our sign convention in KCF, singular values are \underline{minus} eigenvalues, i.e. $- \Lambda(\mathcal{L})$.

\begin{definition} Consider a pencil of skew-symmetric forms $\left\{ A_{\lambda} = A + \lambda B\right\}$.

\begin{enumerate}

\item The \textbf{core} subspace is \[ K = \sum_{\lambda - regular} \operatorname{Ker} A_{\lambda}. \] 

\item The \textbf{mantle} subspace is the skew-orthogonal complement to the core (w.r.t. any regular form $A_{\mu}$) \[ M = K^{\perp}. \] 

\end{enumerate}

\end{definition} 

Now fix any basis from the JK theorem. Denote by $V_J$ and $V_K$ the sum of all Jordan and all Kronecker blocks respectively. Using the JK theorem it is easy to prove the following statement.
\begin{assertion} \label{A:CoreMantle} For any JK decomposition we have the following. 

\begin{enumerate}

\item  The core subspace K is spanned by vectors corresponding to the down-right zero matrices of Kronecker blocks, like this one:  \[ A_i + \lambda B_i = \left(
\begin{array}{c|c}
  0 & \begin{matrix}
   1 & \lambda      &        &     \\
      & \ddots & \ddots &     \\
      &        & 1    & \lambda  \\
    \end{matrix} \\
  \hline
  \begin{matrix}
  \minus1  &        &    \\
  \minus \lambda   & \ddots &    \\
      & \ddots & \minus1 \\
      &        & \minus\lambda  \\
  \end{matrix} &\cellcolor{blue!25} 0 
 \end{array}
 \right).
 \]

\item The mantle subspace is the core plus all Jordan blocks: \[ M = K \oplus V_J. \]

\end{enumerate}

\end{assertion}

\subsection{Admissible subspaces}

Let $\mathcal{L}= \left\{A_{\lambda} \right\}$ be a linear pencil on $V$. For a subspace $U\subset (V, \mathcal{L})$ we denote by $U^{\perp_{A_\lambda}}$ or $U^{\perp_{\lambda}}$ its skew-orthogonal complement w.r.t. the form $A_{\lambda}$: \[ U^{\perp_{\lambda}} = \left\{ v\in V \, \, \bigr| \, \, A_{\lambda}(v, U) = 0\right\}.\]

\begin{definition} A subspace $U \subset  (V, \mathcal{L})$ is \textbf{admissible} if its skew-orthogonal complements $U^{\perp_{A_\lambda}}$ coincide for almost all forms $A_\lambda$ of the pencil $\mathcal{L}$. We denote this complement as $U^{\perp_{\mathcal{L}}}$ or $U^{\perp}$. \end{definition}

Note that ``admissibility'' of $U$ does not depend on ``the choice of basis'' in the pencil $\mathcal{L}$.

\begin{assertion} \label{A:AsserBasChange} Let $U \subset V$ be an admissible space for a pencil $\mathcal{L} = \left\{ A + \lambda B\right\}$. Then for any linearly independent forms $A', B' \in \operatorname{span}\left\{A, B\right\}$ the subspace $U$ is also admissible w.r.t. the pencil $\mathcal{L}' =\left\{A' + \lambda B' \right\}$.
\end{assertion}

In the Jordan case (i.e. when the regular forms are nondegenerate) admissible subspaces $U$ have a simple description. Simply speaking, ``admissible = $P$-invariant''.

\begin{assertion}[{\cite[Assertion 2.4]{Kozlov24BiLagr}}] \label{A:AdmChains}
Let $\mathcal{L} = \left\{ A + \lambda B\right\}$ be a linear pencil on $V$, $B$ be a regular form and $P = B^{-1}A$ be the recursion operator. A subspace $U \subset (V, \mathcal{L})$ is admissible if and only if $U$ is $P$-invariant.
\end{assertion}

 We also have the following statements about any bi-Poisson linear space $(V, \mathcal{L})$. 
 
\begin{assertion} \label{A:SumAdm}  Any sum of admissible subspaces $\oplus_{\alpha} U_{\alpha}$ is also admissible.

\end{assertion}

The next statement can be proved using the JK theorem.

\begin{lemma} \label{L:AdmInMantle} The core subspace $K$ and the mantle subspace $M$ are admissible. The mantle subspace $M$ is the maximal admissible subspace, i.e. any other admissible space $U \subset M$. \end{lemma}

By Lemma~\ref{L:AdmInMantle}  and Assertion~\ref{A:AdmChains} a subspace $U$ between the core and the mantle \[ K \subset U \subset M \]
is admissible if and only if $U/K$ is $P$-invariant in $M/K$ (where $P$ is the induced recursion operator). Eigenvectors of $P$ span a $P$-invariant subspace. Hence, we get the following simple statement that we need below.

\begin{assertion} \label{A:CommutLinear} Let $\mu_1, \dots, \mu_D \in \mathbb{C} \cup \left\{\infty\right\}, D \in \mathbb{N}$ be any distinct values,  $v_i \in \operatorname{Ker} (A + \mu_i B), i=1, \dots, D$ be any vectors. The subspace \[ U = K + \operatorname{span}\left\{ v_1, \dots, v_D \right\},\] where $K$ is the core subspace, is bi-isotropic and admissible.
\end{assertion}

\subsection{Bi-Lagrangian subspaces}

Bi-Lagrangian subspaces were extensively studied in   \cite{Kozlov24BiLagr}.

\begin{definition} A subspace $U \subset V$ of a bi-Poisson vector space $(V, \mathcal{L})$ is called 

\begin{itemize}

\item \textbf{bi-isotropic} if $A_{\lambda}(u, v) = 0$ for all $u, v\in V$ and all $A_{\lambda} \in \mathcal{L}$;

\item  \textbf{bi-Lagrangian} if it is bi-isotropic and $\dim U = \dim V - \frac{1}{2} \operatorname{rk} \mathcal{L}$.

\end{itemize}

 \end{definition}

\begin{assertion}[{\cite[Lemma 3.2]{Kozlov24BiLagr}}] Any bi-Lagrangian subspace $L \subset (V,\mathcal{L})$ contains the core subspace K
and is contained in the mantle subspace $M$: \[ K \subset L \subset M.\] \end{assertion}

Below we need the following statement that easily follows from the Jordan--Kronecker theorem.

\begin{assertion} \label{A:DualCoreLagr} Let $(V, \mathcal{L})$ be a linear bi-Poisson space with the core subspace $K$ and the mantle subspace $M$. 

\begin{enumerate}

\item The annihilator of the core subspace is $\displaystyle K^0 = \bigcap_{\lambda \text{ - reg.}} \operatorname{Im} A_{\lambda}$. 

\item For any bi-Lagrangian subspace $L$, since $K \subset L$, we have $\displaystyle L^0 \subset \bigcap_{\lambda \text{ - reg.}} \operatorname{Im} A_{\lambda}$.

\item For any $\alpha \in \bar{\mathbb{C}}$ we have $\displaystyle A_{\alpha}^{-1}(K^0) \subset M$.

\end{enumerate}

\end{assertion}

\section{Poisson pencils}

In this section we introduce some essential definitions and notions associated with Poisson pencils. Let $\mathcal{P} = \left\{ \mathcal{A}_{\lambda} = \mathcal{A} + \lambda \mathcal{B}\right\}$  be a Poisson pencil on $M$. The \textbf{rank} of $\mathcal{P}$ is given by \eqref{Eq:PencilRank}. Similarly, the rank of $\mathcal{P}$ at a point $x \in M$ is \[  \operatorname{rk} \mathcal{P}(x) = \max_{\lambda \in \bar{\mathbb{K}}} \operatorname{rk}\mathcal{A}_{\lambda}(x).\] A bracket $A_{\lambda} \in \mathcal{P}$ is \textbf{regular at a point $x$} if \[ \operatorname{rk} \mathcal{A}_{\lambda}(x) = \operatorname{rk} \mathcal{P}(x).\]  To exclude singularities of the pencil $\mathcal{P}$, we will consider the following points $x\in M$.  \begin{definition} \label{D:JKReg} A point $x_0 \in (M,\mathcal{P})$ is \textbf{JK-regular} if in a neighborhood of $Ox_0$ the pencils $\mathcal{P}(x)$ have the same Kronecker Canonical Form, up to the eigenvalues\footnote{Some authors say that $\mathcal{P}(x)$ belong to the same bundle or have the same algebraic type. ``Algebraic type'' and ``bundle of a linear pencil'' is roughly the same thing.}. \end{definition}

In other words, $x_0\in (M, \mathcal{P})$  is JK-regular if in a  neighborhood of $x_0$ there exists a local frame $v_1(x), \dots,
v_n(x)$  such that the matrices of $\mathcal{A}$ and $\mathcal{B}$ have the block-diagonal form as in the  JK theorem, but the eigenvalues $\lambda_i(x)$ depend on $x\in M$: \begin{equation} \renewcommand*{\arraystretch}{1.2} A_i = \left(\begin{array}{c|c} 0 & J(\lambda_i(x)) \\
\hline - J^{T}(\lambda_i(x)) &
0\end{array} \right), \quad  B_i = \left(\begin{array}{c|c} 0 & E \\
\hline - E& 0\end{array} \right).\end{equation} Note that for JK-regular points the number of distinct eigenvalues $\lambda_i(x)$ locally remains the same. Eigenvalues that are equal at $x_0$ remain equal in a neighborhood $Ox_0$:\[ \lambda_i(x_0) = \lambda_j(x_0) \qquad \Rightarrow \qquad \lambda_i(x) = \lambda_j(x), \quad x \in Ox_0.\] 

In \cite{BolsZhang} the \textbf{characteristic polynomial} $p_{\mathcal{P}}(\lambda)$ of $\mathcal{P} = \left\{ \mathcal{A} + \lambda \mathcal{B}\right\}$ is defined as follows. Consider all diagonal minors $\Delta_I$ of the matrix $\mathcal{A} + \lambda \mathcal{B}$ of order rank $\mathcal{P}$ and take the Pfaffians $\operatorname{Pf}(\Delta_I)$, i.e. square roots, for each of them. The characteristic polynomial is the greatest common divisor of all these Pffaffians: \[ p_{\mathcal{P}} = \operatorname{gcd} \left(\operatorname{Pf}(\Delta_I) \right). \]

\subsection{Constructing new Poisson pencils using Casimir functions}

A function $f$ is a \textbf{Casimir function} of a Poisson bracket $\mathcal{A}$ if $\mathcal{A} df = 0$. We denote the set of all Casimir functions associated with a Poisson bracket $\mathcal{A}$ as $\mathcal{C}\left( \mathcal{A}\right)$.

\begin{assertion} \label{A:NewPoissonCasimir} Let $\mathcal{A}$ and $\mathcal{B}$ be two compatible Poisson brackets on $M$. Assume that $f$ is a Casimir function for both brackets, i.e. $f \in \mathcal{C} \left( \mathcal{A} \right) \cap \mathcal{C}\left(\mathcal{B}\right)$. Then we have the following:

\begin{enumerate}
    \item The sum $\mathcal{A}_f = \mathcal{A} + f \mathcal{B}$ is a well-defined Poisson bracket on $M$.
    
    \item The bracket $\mathcal{A}_f$ is compatible with the brackets $\mathcal{A}$ and $\mathcal{B}$.

    \item The KCF of $\mathcal{A}_f(x) + \lambda \mathcal{B}(x)$ can be obtained from KCF of $\mathcal{A}(x) + \lambda \mathcal{B}(x)$ if we replace each eigenvalue $\lambda_j(x)$ with $\lambda_j(x) + f(x)$.

    \item Functions $g$ and $h$ are in bi-involution w.r.t. $\mathcal{A}$ and $\mathcal{B}$ if and only if they are in bi-involution w.r.t. $\mathcal{A}_f$ and $\mathcal{B}$.
\end{enumerate} 

\end{assertion}

\begin{proof}[Proof of Assertion~\ref{A:NewPoissonCasimir}] All the statement can be directly derived from the following well-known facts about the Schouten bracket, also known as the \textbf{Schouten--Nijenhuis bracket} (see e.g. \cite{DufourZung05}):

\begin{itemize}

\item Two Poisson brackets $\mathcal{A}$ and $\mathcal{B}$ are compatible if and only if their Schouten bracket vanishes $[\mathcal{A}, \mathcal{B}] = 0$.

\item A function $f$ is a Casimir function of a Poisson bracket $\mathcal{A}$ if and only if  their Schouten bracket vanishes $[f, \mathcal{A}] = 0$.

\end{itemize}

Assertion~\ref{A:NewPoissonCasimir} is proved. \end{proof}

\subsection{Core distribution} \label{S:CoreMantle}

A \textbf{distribution} on a manifold $M$ is the assignment to each point $x$ of M a vector subspace $D_x$ of the tangent space $T_xM$. The dimension of $D_x$ may depend on $x$. 

\begin{example} If $\mathcal{F} = \left\{ f_{\alpha} \,\, \bigr|\,\, \alpha \in  A\right\}$ is a family of functions on manifold $M$, then by $d \mathcal{F}$ denote the distribution in $T^*M$ given by  \[ d\mathcal{F}(x) = \operatorname{span} \left\{ df_{\alpha}(x) \, \, \bigr|\,\, \alpha \in A \right\}. \] \end{example}

For any distribution $\Delta \subset TM$ we can also consider its dual distribution $\Delta^0 \subset T^*M$, which is the distribution of annihilators.  We say that a distribution or subbundle $\Delta \subset T^*M$ is \textbf{isotropic} (bi-isotropic, etc.) if each subspace $\Delta_x$ is isotropic (bi-isotropic, etc.). The next statement is trivial.

\begin{assertion} \label{A:BiLagr} A family of function $\mathcal{F}$ on $(M,\mathcal{P})$ is a complete family of functions in bi-involution if and only if $d \mathcal{F}(x)$ is bi-Lagrangian subspace at a generic point $x \in M$. \end{assertion}

In this section we discuss the following important distribution.

\begin{definition} \label{D:CoreMantleDist} Let $\mathcal{P} = \left\{ \mathcal{A}_{\lambda} = \mathcal{A} + \lambda \mathcal{B}\right\}$ be a Poisson pencil on $M$.  The core subspace in each cotangent space $T^*_x M$ defines a  the \textbf{core distribution} $\mathcal{K}$ in $T^*M$. In other words, at each point $x \in M$
 \begin{equation} \label{Eq:CoreDist} \mathcal{K}_x =  \bigoplus_{\lambda - \text{regular for $\mathcal{P}(x)$}} \operatorname{Ker}\mathcal{A}_\lambda (x),  \end{equation}  \end{definition}

In practice we can generate the core distribution by taking a sufficient number of (local) Casimir functions. The next statement easily follows from the Jordan--Kronecker theorem and the Darboux--Weinstein Theorem.

\begin{assertion} \label{A:LocSpan} Let $\mathcal{P}$ be a Poisson pencil with  on $M$. If $\operatorname{rk} \mathcal{P} = \operatorname{const}$ on $M$, then in a sufficiently small neighborhood $U$ of any point $x_0$ there exist Casimir functions $f_{j, 1}, \dots, f_{j, m_j} \in \mathcal{C}\left(\mathcal{A}_{\mu_j}\right), j = 1, \dots, D$ such that 

\begin{enumerate} 

\item  $\mathcal{A}_{\mu_j}(x), j=1,\dots, D$ are regular in the linear pencil $\mathcal{P}(x)$ for any $x \in U$;  

\item the core distribution $\mathcal{K}$ is locally spanned by  the differentials of  Casimir functions:

\[ \mathcal{K}_x =  \operatorname{span} \left\{  df_{1, 1}(x), \dots, df_{D, m_D}(x) \right\}, \qquad \forall x \in U. \]

\end{enumerate}

\end{assertion}

\section{Bi-Poisson reduction} \label{S:BiPoisRed}

Bi-Poisson reduction is the fundamental technique that enables us to prove bi-integrability of bi-Hamiltonian systems. The main result is Theorem~\ref{T:BiPoissRed} in Section~\ref{SubS:BiPoisRed}. As a preliminary step, we present a linear analogue of bi-Poisson reduction for linear pencils in Section~\ref{S:BiPoisRedLin}. In Section~\ref{SubS:Caratheodory} we establish some technical results that we use in the proof of Theorem~\ref{T:BiPoissRed}.

\subsection{Linear bi-Poisson reduction} \label{S:BiPoisRedLin}

The next theorem is an analogue of linear symplectic reduction for a pair of $2$-forms. 

\begin{theorem} \label{T:BiPoissReduction} Let $\mathcal{L} = \left\{A_{\lambda} \right\}$ be a linear pencil on $V$ and let $U\subset \left(V, \mathcal{L} \right)$ be an admissible bi-isotropic subspace. Then

\begin{enumerate}

\item The induced pencil $\mathcal{L} ' = \left\{A'_{\lambda}\right\}$ on $U^{\perp}/ U$ is well-defined. 

\item If $L$ is a bi-Lagrangian (or bi-isotropic) subspace of $(V, B)$, then \[ L' = \left( \left( L \cap U^{\perp}\right) + U \right) / U\] is a bi-Lagrangian (respectively, bi-isotropic) subspace of $U^{\perp}/U$. 

\end{enumerate}

\end{theorem}

We need the following simple statement. 

\begin{assertion} \label{A:AdmSpectr}  Under the conditions of Theorem~\ref{T:BiPoissReduction}, if the admissible subspace $U$ contains the core subspace $K$, then the following holds.

\begin{enumerate} 

\item All eigenvalues of $\mathcal{L} '$ are eigenvalues of $\mathcal{L} $, i.e. \begin{equation} \label{Eq:SpectrSub} \sigma(\mathcal{L} ') \subseteq \sigma(\mathcal{L} ).\end{equation} In other words, if $A_{\lambda} \in \mathcal{L} $ is regular, then the induced form $A'_{\lambda}$ is also regular.

\item The induced pencil  $\mathcal{L} ' = \left\{A'_{\lambda}\right\}$ is nondegenerate, i.e. $\operatorname{Ker}\mathcal{A}'_{\lambda} = 0$ for generic $\lambda$.

\end{enumerate}
\end{assertion}

\begin{proof}[Proof of Assertion~\ref{A:AdmSpectr}] In the Jordan case the subspace $U^{\perp}/U$ is $P$-invariant and the induced pencil on it is nondegenerate and doesn't have new eigenvalues. We can reduce the general case to the Jordan case by performing the reduction as in Theorem~\ref{T:BiPoissReduction} w.r.t. the core subspace $K$. Assertion~\ref{A:AdmSpectr} is proved. 
\end{proof}

\subsection{Caratheodory--Jacobi--Lie theorem for Poisson manifolds} \label{SubS:Caratheodory}

In Section~\ref{SubS:BiPoisRed} we establish  integrability of the subbundle $\Delta^{\perp} \subset T^*M$, where $\Delta \subset T^*M$ is an integrable bi-isotropic admissible subbundle that contains the core distribution $\mathcal{K} \subset \Delta$.  To achieve this, we will utilize the following Caratheodory--Jacobi--Lie theorem for Poisson manifolds. It is a slight modification of \cite[Theorem 2.1]{Miranda08}.

\begin{theorem} \label{T:CaraJacobLiePoisson} Let $(M, \mathcal{A})$ be a Poisson manifold, $\dim M = n$ and $\operatorname{rk} \mathcal{A} = 2k$ on $M$. Assume that 

\begin{itemize} 

\item  $z_1, \dots, z_{n-2k}$ are Casimir functions, i.e. $\left\{f, z_j\right\} = 0$,

\item  $p_1, \dots, p_r$, where $r \leq k$, are smooth functions in involution $\left\{p_i, p_j \right\} = 0$,

\item $dp_1, \dots, dp_r$ and $dz_1,\dots, dz_{n-2k}$ are linearly independent at $x \in M$, i.e. \[ \left(dp_1 \wedge \dots \wedge dp_r \wedge dz_1 \wedge \dots dz_{n-2k}\right)\bigr|_{x} \not = 0.\] 

\end{itemize} Then there exist functions $p_{r+1}, \dots, p_k, q_1, \dots, q_k$ such that $(p_i, q_i, z_j)$ are local coordinates  at $x$ and  \begin{equation} \label{Eq:CaraJracobLieBracket}\mathcal{A} = \sum_{i=1}^k \frac{\partial}{\partial p_i} \wedge \frac{\partial}{\partial q_i}. \end{equation} \end{theorem}

\begin{proof}[Proof of Theorem~\ref{T:CaraJacobLiePoisson}] Since $dp_1, \dots, dp_r, dz_1,\dots, dz_{n-2k}$ are linearly independent and \[\operatorname{Ker} \mathcal{A} = \operatorname{span}\left\{dz_1, \dots, dz_{n-2k} \right\}\] the Hamiltonian vector fields $X_{p_1}, \dots, X_{p_r}$ are linearly independent. By \cite[Theorem 2.1]{Miranda08} there exists local coordinates $p_1, \dots, p_r, q_1, \dots, q_r, s_1, \dots, s_{n-2r}$ such that \[ \mathcal{A} =  \sum_{i=1}^r \frac{\partial}{\partial p_i} \wedge \frac{\partial}{\partial q_i} + \sum_{i, j = 1}^{n-2r} g_{ij}(s)  \frac{\partial}{\partial s_i} \wedge \frac{\partial}{\partial s_j}. \] It remains to note that $z_j$ are Casimir functions for the Poisson bivector \[\sum_{i, j = 1}^{n-2r} g_{ij}(s)  \frac{\partial}{\partial s_i} \wedge \frac{\partial}{\partial s_j}\] and apply the Darboux--Weinstein theorem for it. Theorem~\ref{T:CaraJacobLiePoisson} is proved. \end{proof}

We need Theorem~\ref{T:CaraJacobLiePoisson} for the following statement. Recall that a subbundle $\Delta \subset T^*M$ is integrable if and only if its dual subbundle $\Delta^0 \subset TM$ is integrable. 

\begin{corollary} \label{Cor:PoisDualIntDistInt} Let  $(M, \mathcal{A})$ be a Poisson manifold and $\operatorname{rk} \mathcal{A} = 2k$ on $M$. Let $\Delta \subset T^*M$ be an integrable isotropic subbundle such that  $\operatorname{Ker} \mathcal{A} \subset \Delta$. Then $\Delta^{\perp}$ is an integrable subbundle of $T^*M$.  \end{corollary}

\begin{proof}[Proof of Corollary~\ref{Cor:PoisDualIntDistInt} ] $\Delta^{\perp}$ is a subbundle, since $\operatorname{Ker} \mathcal{A} \subset \Delta$ and $\operatorname{rk} \mathcal{A} = \operatorname{const}$. It remains to prove the integrability of $\Delta^{\perp}$. Take any point $x \in M$.

\begin{itemize}

\item  Let $z_1, \dots, z_{n-2k}$ be local Casimir functions at $x$, i.e. $\left\{f, z_j\right\} = 0$,. 

\item Since $\Delta$ is integrable and  $\operatorname{Ker} \mathcal{A} \subset \Delta$, there exists functions $p_1, \dots, p_r$ such that $dp_1, \dots, dp_r$ and $dz_1,\dots, dz_{n-2k}$ are linearly independent at $x \in M$, i.e. \[ \left(dp_1 \wedge \dots \wedge dp_r \wedge dz_1 \wedge \dots dz_{n-2k}\right)\bigr|_{x} \not = 0,\]and $\Delta$  is locally given by the level sets of the functions $p_i, z_j$, i.e. \[ \Delta = \operatorname{span} \left\{ dp_1, \dots, dp_r,  dz_1, \dots, dz_{n-2k}\right\}.\] 

\item Since $\Delta$ is isotropic, the functions  $p_1, \dots, p_r$ are in involution $\left\{p_i, p_j \right\} = 0$.

\end{itemize}

Thus, we can apply  Theorem~\ref{T:CaraJacobLiePoisson} and get local coordinates \[p_1, \dots, p_k, q_1, \dots, q_k, z_1, \dots, z_{n-2k}\] such that \eqref{Eq:CaraJracobLieBracket} holds. In this coordinates \[\Delta^{\perp} = \operatorname{span}\left\{ dp_1, \dots, dp_k, dq_{r+1}, \dots, dq_k, dz_1, \dots, dz_{n-2k}\right\}.\] Therefore, $\Delta^{\perp}$ is integrable. Corollary~\ref{Cor:PoisDualIntDistInt} is proved. \end{proof}

\begin{remark} In the holomorhpic case the proof remains the same, but one should use holomorphic analogues of some theorems. For instance, instead of the Frobenious theorem one can use the fact that involutive holomorphic subbundles are integrable in the holomorphic sense (see e.g. \cite{Voisin}). \end{remark}

\subsection{Bi-Poisson reduction theorem} \label{SubS:BiPoisRed}

The next result is the main technique that allows us to bi-integrate bi-Hamiltonian systems. This theorem was previously established for $\Delta = \mathcal{K}$  in \cite[Theorem 5.9]{Kozlov23JKRealization}, the proof for the general case is roughly the same.

\begin{theorem}  \label{T:BiPoissRed}  Let $\mathcal{P} = \left\{ A_{\lambda} = \mathcal{A} + \lambda \mathcal{B}\right\}$ be a Poisson pencil on $M$ such that $\operatorname{rk} \mathcal{\mathcal{P}}(x) = 2k$ for all $x \in M$. Let $\Delta \subset T^*M$ be an integrable bi-isotropic admissible subbundle that contains the core distribution $\mathcal{K} \subset \Delta$. Then the following holds:

\begin{enumerate} 

\item $\Delta^{\perp}$ is an integrable admissible subbundle of  $T^*M$. 

\item Moreover, there exist local coordinates \begin{equation} \label{Eq:LocCoorBiPoissRed} (p, f, q) = (p_1,\dots, p_{m_1}, f_1, \dots, f_{m_2}, q_1, \dots, q_{m_3})\end{equation} such that \begin{equation} \label{Eq:DDperpleDistLoc} \Delta = \operatorname{span}\left\{dq_1, \dots, dq_{m_3} \right\}, \quad \Delta^{\perp} = \operatorname{span}\left\{df_1, \dots, df_{m_2}, dq_1, \dots, dq_{m_3} \right\}\end{equation} and the pencil has the form \begin{equation} \label{Eq:BiPoissMat} \mathcal{A}_{\lambda} = \sum_{i=1}^{m_1} \frac{\partial}{\partial p_i} \wedge v_{\lambda, i} + \sum_{1 \leq i < j \leq m_2} c_{\lambda, ij}(f, q)  \frac{\partial}{\partial f_i} \wedge \frac{\partial}{\partial f_j} \end{equation} for some vectors $v_{\lambda, i} = v_{\lambda, i}(p, f, q)$ and some functions $c_{\lambda, ij}(f, q)$. 

\end{enumerate}

\end{theorem}

Simply speaking, the matrices of the Poisson brackets in Theorem~\ref{T:BiPoissRed} take the form \[\mathcal{A}_{\lambda} = \left( \begin{matrix} * & * & * \\ * & C_{\lambda}(f, q) & 0 \\ * & 0 & 0\end{matrix} \right), \]  where $*$ are some matrices.  Obviously, the vector fields $v_{\lambda, i} = v_{\lambda, i}(x, s, y)$ and the functions $c_{\lambda, ij}(s, y)$ depend linearly on $\lambda$: \[ v_{\lambda, i} = v_{0, i} + \lambda v_{\infty, i}, \qquad c_{\lambda, ij}(f, q) = c_{0, ij}(f, q) + \lambda c_{\infty, ij}(f, q).\]

\begin{proof}[Proof of Theorem~\ref{T:BiPoissRed}]  $\Delta^{\perp}$ is a  integrable admissible subbundle by Corollary~\ref{Cor:PoisDualIntDistInt}.  The rest of the proof is in several steps.

\begin{enumerate} 

\item  Since $\Delta$ and $\Delta^{\perp}$ are integrable, there exist local coordinates \eqref{Eq:LocCoorBiPoissRed} such that $\Delta$ and $\Delta^{\perp}$ have the form \eqref{Eq:DDperpleDistLoc}. 

\item Since $\Delta$ is bi-isotropic, $\Delta \subset \Delta^{\perp}$ and thus \[ \left\{f_i,q_j \right\}_{\lambda} = 0, \qquad \left\{ q_i, q_j \right\}_{\lambda} = 0.\] In other words,  the matrices of Poisson brackets have the form \[\mathcal{A}_{\lambda} = \left( \begin{matrix} * & * & * \\ * & C_{\lambda}(p, f, q) & 0 \\ * & 0 & 0\end{matrix} \right). \]

\item It remains to prove that $c_{\lambda, ij} = \left\{f_i, f_j\right\}_{\lambda}$ do not depend on $p_1, \dots, p_{m_1}$. It follows from the Jacobi identity: \[\left\{ q_k, \left\{ f_i, f_j \right\}_{\lambda} \right\}_{\lambda} = \left\{ \left\{ q_k, f_i, \right\}_{\lambda}  f_j \right\}_{\lambda} +\left\{ f_i, \left\{ q_k,  f_j \right\}_{\lambda}  \right\}_{\lambda} = 0.\]  Consider the Hamiltonian vector fields \[\mathcal{A}_{\lambda} dq_k  = \left\{ q_k, \cdot \right\}_{\lambda}.\] Recall that $\Delta = \operatorname{span} \left\{dq_1, \dots, dq_{m_3}\right\}$ contain the core $\mathcal{K}$. Using the JK theorem,  it is easy to check that for any point $x \in M$ and for any value $\lambda \in \bar{\mathbb{C}}$ that is regular for $\mathcal{P}(x)$ we have \[ \operatorname{span} \left\{ \mathcal{A}_{\lambda} dq_1, \dots, \mathcal{A}_{\lambda} dq_{m_3}\right\} (x) = \left\{ \frac{\partial }{\partial  p_1}, \dots, \frac{\partial }{\partial  p_{m_1}}\right\}.\]  We get that \[ \frac{\partial \left\{ f_i, f_j \right\}_{\lambda} }{\partial  p_k}  = 0, \qquad k =1, \dots, m_1, \qquad \forall  \lambda\in \bar{\mathbb{C}}\] and thus $c_{\lambda, ij} = c_{\lambda, ij}(f, q)$, as required.

\end{enumerate}

Theorem~\ref{T:BiPoissRed} is proved.  \end{proof}

\begin{definition} Let $\mathcal{P}$ be a Poisson pencil on $M$ with constant rank and $\Delta \subset T^*M$ be an integrable bi-isotropic admissible subbundle. We perform a local \textbf{bi-Poisson reduction} near $x \in M$ by quotienting a sufficiently small neighborhood $U$ of $x$ by the distribution $\left(\Delta^{\perp}\right)^0$. This induces a new Poisson pencil $\mathcal{P}'$ on the quotient space $U/ \left(\Delta^{\perp}\right)^0$,  with the projection \[ \pi: (U, \mathcal{P}) \to \left(U/ \left(\Delta^{\perp}\right)^0, \mathcal{P}' \right).\] \end{definition}

Theorem~\ref{T:BiPoissRed} guarantees that we can perform (local) bi-Poisson reduction. In the local coordinates $(p, f, q)$ from this theorem \[ \left(\Delta^\perp\right)^0 = \operatorname{span}\left\{\frac{\partial}{\partial p_1}, \dots, \frac{\partial}{\partial p_{m_1}}\right\}.\] Thus, $(f,q)$ are local coordinates on the quotient $U/ \left(\Delta^{\perp}\right)^0$ and the induced pencil $\mathcal{P}'$ takes the form \[ \mathcal{P}' = \left(\begin{matrix}  C_{\lambda}(f, q) & 0 \\ 0 & 0 \end{matrix} \right).\] 

\begin{remark} In practice, we often consider a set of functions $\mathcal{F}$ in bi-involution and $\Delta = d \mathcal{F}$. We then quotient by the local action of the Hamiltonian vector fields: \[  \left(\Delta^{\perp}\right)^0 = \left\{ \mathcal{A}_{\lambda} dg  \quad \bigr| \quad g \in \mathcal{F} \right\}, \] for any  Poisson bracket $\mathcal{A}_{\lambda} \in \mathcal{P}$ that is regular everywhere on $M$. When performing bi-Poisson reduction, we essentially focus on the algebra of functions that are in bi-involution with all functions in $\mathcal{F}$. \end{remark}

\section{Main theorem} \label{S:MainTh}

Our objective is to prove Theorem~\ref{T:MultHamBiInt}. In fact, we prove a more general Theorem~\ref{T:MainGenTh}.  Let $\mathcal{P} = \left\{\mathcal{A}_{\lambda} \right\}$ be a Poisson pencil on $M$ and $v = \mathcal{A}_{\lambda} dH_{\lambda}$ be a system which is bi-Hamiltonian w.r.t. $\mathcal{P}$.

\begin{definition} \label{Def:StandInt} For an open subset $U \subset M$ the \textbf{family of standard integrals} $\mathcal{F}$ on $U$ consists of the following functions:

\begin{enumerate}
    \item Casimir functions $f_{\lambda}$ for brackets $\mathcal{A}_{\lambda}$ that are regular on $U$.

    \item Eigenvalues $\lambda_j(x)$ of the pencil $\mathcal{P}$. 
    
    \item The Hamiltonians $H_{\alpha}$ for all $\alpha \in \bar{\mathbb{K}}$. 
\end{enumerate}

\end{definition}

We prove that the standard integrals are first integrals of a bi-Hamiltonian system in Section~\ref{SubS:StandBiInv}. Casimir functions and eigenvalues may not be well-defined on the entire manifold $M$.  Therefore, we restrict our attention to a sufficiently small neighborhood\footnote{Alternatively, one could consider the germs of these local integrals.} of a point $x \in M$ to ensure their well-definedness.

\begin{definition} \label{Def:SmallNeigh} We say that a neighborhood $Ux$ of a point $x\in M$ is \textbf{small} if the following two conditions hold:

\begin{enumerate}
    \item The core distribution $\mathcal{K} \subseteq d \mathcal{F}$, where $\mathcal{F}$ is the family of standard integrals on $Ux$. In other words, $Ux$ satisfies Assertion~\ref{A:LocSpan}. 
    
    \item All eigenvalues $\lambda_j$ are finite, i.e. $\lambda_j < \infty$, and are well-defined functions on $Ux$.
\end{enumerate} 

\end{definition}

It is evident that any JK-regular point $x\in M$ that possesses finite eigenvalues $\lambda_j(x) < \infty$ has a small neighborhood. Our main result is the following.

\begin{theorem} \label{T:MainGenTh} Let $\mathcal{P}$ be a Poisson pencil on a real $C^{\infty}$-smooth or complex analytic manifold $M$ and $v = \mathcal{A}_\lambda d H_{\lambda}$ be a vector field that is bi-Hamiltonian w.r.t. $\mathcal{P}$. In the real case, we assume that all eigenvalues of $\mathcal{P}$ are real. Let $x_0 \in M$ be a JK-regular point and $\mathcal{F}$ be a family of standard integrals in a small neighborhood $Ux_0$. Assume that the following two conditions are satisfied:

\begin{enumerate} 

\item Locally, within a neighborhood of $x_0$, \[\dim d \mathcal{F}(x) = \operatorname{const}.\]

\item After bi-Poisson reduction w.r.t. $d \mathcal{F}$ the point $x_0$ remains JK-regular. 

\end{enumerate}

Then in a sufficiently small neighborhood $Ox_0$ the family $\mathcal{F}$ can be extended to complete family of functions $\mathcal{G} \supset \mathcal{F}$ in bi-involution. \end{theorem}

The family $\mathcal{G}$ from Theorem~\ref{T:MainGenTh} consists of  first integrals of the system by the following simple statement.

\begin{assertion} \label{A:FirstInt} Consider a Hamiltonian system $v = \mathcal{A} dH$ and a commutative family of functions $\mathcal{G}$ containing the Hamiltonian $H$. Then, every function in  $\mathcal{G}$ is a first integral of the system. \end{assertion}

\begin{proof}[Proof of Assertion~\ref{A:FirstInt}] For any function $f \in \mathcal{G}$ we have \[v(f) = \left\{f, H\right\}_{\mathcal{A}} = 0. \]  Assertion~\ref{A:FirstInt} is proved. \end{proof}

\subsection{Standard integrals are in bi-involution} \label{SubS:StandBiInv}

Standard integrals were defined in Definition~\ref{Def:StandInt}.

\begin{lemma} \label{L:StandIntBiInv} The family of standard integrals $\mathcal{F}$ on $M$ is in bi-involution. \end{lemma}

In order to prove Lemma~\ref{L:StandIntBiInv} we use an important fact about eigenvalues of Poisson pencils on a manifold $M$.   Although we were not able to find the statement of Lemma~\ref{L:EigenDiff} in
the literature, it is well-known to the experts in the field. For nondegenerate pencils Lemma~\ref{L:EigenDiff} follows from a similar statement about eigenvalues of Nijenhuis operators
(see \cite[Proposition 2.3]{BolsinovNijenhuis}). For the proof see e.g. \cite[Lemma 9.8]{Kozlov23JKRealization}.

\begin{lemma} \label{L:EigenDiff} Let $\mathcal{P} = \left\{ \mathcal{A} + \lambda \mathcal{B} \right\}$ be a Poisson pencil on a manifold $M$. For any JK-regular point $x \in (M, \mathcal{P})$ and any finite eigenvalue $\lambda_j(x) < \infty$ we have \begin{equation} \label{Eq:Eigen1} (\mathcal{A} - \lambda_j (x) \mathcal{B}) d\lambda_j (x) =0.\end{equation} \end{lemma}

\begin{proof}[Proof of Lemma~\ref{L:StandIntBiInv}] The proof is in several steps: \begin{enumerate}

 \item  Casimir functions $f_{\lambda}$ commute (w.r.t. all brackets $\mathcal{A}_{\lambda}$) with other functions from $\mathcal{F}$, since $df_{\alpha} \in \mathcal{K}$ and $d\mathcal{F} \subset \mathcal{M}$, where $\mathcal{K}$ and $\mathcal{M}$ are the core and mantle distributions respectively. $df_{\alpha} \in \mathcal{K}$ by the definition of the core subspace, $d \lambda_j \in \mathcal{M}$ by Lemma~\ref{L:EigenDiff} and $dH_{\alpha} \in \mathcal{M}$ by  Assertion~\ref{A:DualCoreLagr}. 

\item Eigenvalues $\lambda_j(x)$ and $\lambda_k(x)$ commute (w.r.t. all brackets $\mathcal{A}_{\lambda}$) by  Lemma~\ref{L:EigenDiff} and the Jordan--Kronecker theorem.

\item Hamiltonians $H_{\alpha}$ and $H_{\beta}$ commute, since \[ \left\{ H_{\alpha}, H_{\beta}\right\}_{\beta} = v \left(H_{\alpha}\right) =  \left\{ H_{\alpha}, H_{\alpha}\right\}_{\alpha} = 0.\] Similarly, $\left\{ H_{\alpha}, H_{\beta}\right\}_{\alpha} =0$, implying that $\left\{ H_{\alpha}, H_{\beta}\right\}_{\lambda} = 0$ for all $\lambda \in \bar{\mathbb{C}}$.

\item It remains to prove that Hamiltonians $H_{\alpha}$ and eigenvalues $\lambda_j(x)$ commute. The case when $\lambda_j(x) = \operatorname{const}$ (locally) is trivial. It suffices to prove \begin{equation}  \label{Eq:CommHL} \left\{ H_{\alpha}, \lambda_j(x)\right\}_{\lambda} = 0, \qquad \forall \lambda \in \bar{\mathbb{C}}\end{equation} on an open dense subset of $M$. Thus, without loss of generality we can assume that $\alpha \not = -\lambda_j(x)$. On one hand, \[ \left\{ H_{\alpha}, \lambda_j\right\}_{\alpha} = -v (\lambda_j) = 0,\] since $v$ as a bi-Hamiltonian vector field preserves the pencil $\mathcal{P}$ and, therefore, all eigenvalues $\lambda_j(x)$. On the other hand, by Lemma~\ref{L:EigenDiff}, \[  \left\{ H_{\alpha}, \lambda_j\right\}_{-\lambda_j} = 0.\] Since \eqref{Eq:CommHL} hold for two distinct value $\lambda = \alpha$ and $-\lambda_j$ it also hold for all $\lambda \in \bar{\mathbb{C}}$. 
\end{enumerate}

Lemma~\ref{L:StandIntBiInv} is proved. \end{proof}

Consequently, the standard integrals $\mathcal{F}$ are first integrals of the system by Assertion~\ref{A:FirstInt}. 

\subsection{Family of standard integrals is admissible}

\begin{lemma} \label{L:StandIntAdm} The family of standard integrals $\mathcal{F}$ on $M$ is admissible. \end{lemma}

We use the following statement.

\begin{assertion} \label{A:AdmiSubspaceAv} Let $\mathcal{L} = \left\{ A + \lambda B \right\}$ be a linear pencil on a vector space $V$. Let  $v_{\lambda} \in V, \lambda \in \bar{\mathbb{K}}$ be vectors that satisfy \[\beta = A_{\lambda} v_{\lambda}, \] for some fixed covector  $\beta \in V^*$. Then the subspace \[ W = \operatorname{span}\left\{v_{\lambda} \,\, \bigr| \, \, \lambda \in \bar{\mathbb{K}} \right\} + K,\] where $K$ is the core subspace, is admissible. 
\end{assertion}

\begin{proof}[Proof of Assertion~\ref{A:AdmiSubspaceAv}] The proof is in several steps.

\begin{enumerate}

\item \textit{Reduce to the Jordan case.}  By fixing a KCF of $\mathcal{L}$, we can decouple the problem into separate cases for each Jordan and Kronecker block. In the Kronecker case, the JK theorem allows us to easily verify that all $v_{\alpha} \in K$. Hence, in the Kronecker case the subspace $W = K$ is admissible. It remains to consider the Jordan case. 

\item \textit{Without loss of generalization, $A$ and $B$ are regular forms in $\mathcal{L}$.} Indeed, by Assertion~\ref{A:AsserBasChange}, if needed, we can replace $A$ and $B$  with their linear combinations.

\item  \textit{The subspace $W$ is admissible.} Consider the operator $Q= P^{-1} = A^{-1}B$, we need to prove that $W$ is $Q$-invariant. For $\alpha = \infty$ we have \[ Bv_{\infty} = Av_0 \qquad \Leftrightarrow \qquad Qv_{\infty} = v_0. \] For $\alpha \not \in \left\{ 0, \infty\right\}$ we get\[ (\alpha Q + I) v_{\alpha} = v_0 \qquad \Leftrightarrow \qquad Q v_{\alpha} = \frac{1}{\alpha} \left( v_0 - v_{\alpha}\right). \] In the coordinates from the JK theorem it is easy to see that $v_{\alpha}$ is a continuous function on $\alpha$ for regular $\alpha$. Thus, \[Qv_{0} = \lim_{\alpha \to 0} \frac{1}{\alpha} \left( v_0 - v_{\alpha}\right) \in W.\] We see that $W$ is $Q$-invariant, where $Q = P^{-1}$, and thus it is admissible.

\end{enumerate}

Assertion~\ref{A:AdmiSubspaceAv} is proved. \end{proof}

\begin{proof}[Proof of Lemma~\ref{L:StandIntAdm}] 

\begin{itemize}

\item By Definition~\ref{Def:SmallNeigh} and Assertion~\ref{A:LocSpan} differentials of
Casimir functions $f_{\lambda}$ span the core distribution:
\[ \mathcal{K} = \operatorname{span} \left\{ df_{\lambda} \, \bigr| \, f_{\lambda} \in \mathcal{F} \right\}, \] which is an admissible distributions.

    \item By Lemma~\ref{L:EigenDiff} $d \lambda_j(x) \in \operatorname{Ker} \mathcal{A}_{-\lambda_j(x)}$. Thus, by Assertion~\ref{A:CommutLinear} the distribution \begin{equation} \label{Eq:AdmDistEig} \mathcal{K} + \operatorname{span} \left\{ d \lambda_j\right\} = \operatorname{span} \left\{ d\lambda_j, df_{\lambda} \, \bigr| \, \lambda_j, f_{\lambda} \in \mathcal{F} \right\} \end{equation} is admissible. 
    
    \item The distribution \begin{equation} \label{Eq:AdmDistH} \mathcal{K} + \operatorname{span} \left\{ H_{\alpha} \, \bigr| \, \alpha \in \bar{\mathbb{K}} \right\} \end{equation} is admissible by Assertion~\ref{A:AdmiSubspaceAv}.
    
\end{itemize}
 By Assertion~\ref{A:SumAdm} the distribution $d\mathcal{F}$, which is a sum of admissible distributions \eqref{Eq:AdmDistEig} and \eqref{Eq:AdmDistH}, is admissible. Lemma~\ref{L:StandIntAdm} is proved. \end{proof}

\subsection{Proof of Theorem~\ref{T:MainGenTh}}
Our goal is to extend the standard family of functions $\mathcal{F}$ to a complete family of functions $\mathcal{G}$ that are in bi-involution. We begin by setting  $\mathcal{G} = \mathcal{F}$. By Lemmas~\ref{L:StandIntBiInv} and \ref{L:StandIntAdm} $\mathcal{G}$ is an admissible family of functions in bi-involution. By Definition~\ref{Def:SmallNeigh} the core distribution $\mathcal{K} \subset d \mathcal{F}$. The idea of the proof is simple:

\begin{enumerate}

    \item Perform the bi-Poisson reduction w.r.t. $d\mathcal{G}$. 

    \item \label{Step:Extend} Extend $\mathcal{G}$ to a bigger admissible family of functions in bi-involution.
    
    \item Repeat the process until we get enough functions in bi-involution.
    
\end{enumerate}

Since $\mathcal{K} \subset d \mathcal{F} \subset d \mathcal{G}$, after  bi-Poisson reduction w.r.t. $d\mathcal{G}$ the new core distribution $\mathcal{K}_{red} = d\mathcal{G}$ (i.e. all Kronecker blocks are $1 \times 1$). Also, by Assertion~\ref{A:AdmSpectr},  after the reduction all eigenvalues satisfy $d\lambda_j(x) \in \mathcal{K}_{red}$. By Assertion~\ref{A:NewPoissonCasimir}, $\mathcal{A} - \lambda_j(x) \mathcal{B}$ are Poisson brackets. The key question is how to extend the family $\mathcal{G}$ in Step~\ref{Step:Extend}. Our strategy is as follows:

\begin{itemize}
    \item Include a Casimir function of one of these new ``singular'' brackets $\mathcal{A} - \lambda_j(x) \mathcal{B}$.  
\end{itemize} Specifically, we utilize the following statement.

\begin{assertion} \label{A:FuncG} Let $\mathcal{P} = \left\{ \mathcal{A} + \lambda \mathcal{B}\right\}$ be a Poisson pencil on $M$ and $\mathcal{K}$ be its core distribution. Assume that there is an eigenvalue $\lambda(x)$ such that \begin{equation} \label{Eq:LambdaInCore} d\lambda(x) \in \mathcal{K}, \qquad \forall x \in M.\end{equation} Then the following holds:

\begin{enumerate} 

\item For any eigenfunction $g(x)$ such that \begin{equation} \label{Eq:FuncG} \left(\mathcal{A} - \lambda(x) \mathcal{B}\right) dg(x) = 0\end{equation} the distribution \begin{equation} \label{Eq:DistD} \mathcal{D}= \mathcal{K} + \operatorname{span}\left\{dg(x)\right\}\end{equation} is bi-isotropic and admissible. 

\item In a neighborhood of any JK-regular point $x_0 \in M$ there exists a function $g(x)$, given by \eqref{Eq:FuncG}, such that the point $x_0$ remains JK-regular after bi-Poisson reduction w.r.t. the distribution $\mathcal{D}$, given by \eqref{Eq:DistD}.

\end{enumerate}

\end{assertion}

\begin{proof}[Proof of Assertion~\ref{A:FuncG}]

\begin{enumerate}

\item It follows from Assertion~\ref{A:CommutLinear}.

\item $\hat{\mathcal{A}} = \mathcal{A} - \lambda(x) \mathcal{B}$ is a Poisson bracket by Assertion~\ref{A:NewPoissonCasimir}. Since $x_0$ is JK-regular, locally $\operatorname{rk}\hat{\mathcal{A}} = \operatorname{const}$. By the Darboux--Weinstein theorem we can extend any covector $\beta \in \operatorname{Ker}\hat{\mathcal{A}}(x_0) $ to a local Casimir $g(x)$ of  $\hat{\mathcal{A}}$, which is given by \eqref{Eq:FuncG}. 

If a covector $dg(x)$ belongs to a $2m \times 2m$ Jordan $\lambda(x)$-block in a KCF of $\mathcal{P}(x)$, then after bi-Poisson reduction the size of this Jordan block decreases to $(2m-2) \times (2m-2)$.  If we select $\beta$ from the smallest Jordan $\lambda(x_0)$-block in the KCF of $\mathcal{P}(x_0)$, then $dg(x)$ will also belong to a smallest Jordan block locally\footnote{Vectors from the smallest Jordan block satisfy condition of the form $v\not \in \operatorname{Im}P^k$, where $P$ is the recursion operator on $V_J = M/K$. This condition is preserved under small perturbations of $v$.}. This implies that the point $x_0$ remains JK-regular after bi-Poisson reduction.

\end{enumerate}

Assertion~\ref{A:FuncG} is proved.
\end{proof}

The theorem's conditions ensure that $x_0$ is JK-regular after bi-Poisson reduction w.r.t. $d \mathcal{F}$. Hence, we extend the family $\mathcal{G}$ using Assertion~\ref{A:FuncG}. By Theorem~\ref{T:BiPoissReduction}, when the process terminates, the subspaces $d \mathcal{G}(x)$ are bi-Lagrangian subspaces, i.e. $\mathcal{G}$ becomes a complete family of functions in bi-involution.  Theorem~\ref{T:MainGenTh} is proved.


\begin{thebibliography}{99}



\bibitem{BolsinovIzosimomKonyaevOshemkov12} A.\,V.~Bolsinov, A.\,M.~Izosimov, A.\,Y.~Konyaev, A.\,A.~Oshemkov, ``Algebra and topology of integrable systems: research problems'', \textit{Trudi seminara po vectornomu i tenzornomu analizu}, \textbf{28} (2012), 119--191


\bibitem{BolsinovTsonev17}  A.\,V.~Bolsinov, A.\,M.~Izosimov, D.\,M.~Tsonen, ``Finite-dimensional integrable systems: a collection of research problems'', \textit{Journal of Geometry and Physics}, \textbf{115} (2017), 2-15

\bibitem{BolsinovNijenhuis} A.\,V.~Bolsinov, A.\,Yu.~Konyaev, V.\,S.~Matveev, ``Nijenhuis Geometry'', \textit{Advances in Mathematics}, \textbf{394} (2022), 108001




\bibitem{BolsZhang}
A.\,V.~Bolsinov, P.~Zhang, ``Jordan-Kronecker invariants of finite-dimensional Lie algebras'',
\textit{Transformation Groups}, \textbf{21}:1 (2016),  51-86


\bibitem{SilvaWeinstein99}
A.\,Cannas da Silva, A.~Weinstein, \textit{Geometric models for noncommutative algebras}, AMS Berkeley Mathematics Lecture Notes, 10, 1999


\bibitem{DufourZung05} 
J.-P.~Dufour,  N.\,T.~Zung, \textit{Poisson structures and their normal forms}, Progress in Mathematics, Volume 242, Birkhauser Verlag, Basel, 2005

\bibitem{Gantmaher88} F.\,R.~Gantmacher,  {\it Theory of matrices}, AMS Chelsea publishing, 1959

\bibitem{Gelfand79} I.\,M.~Gel’fand, I.\,Ya.~Dorfman, ``Hamiltonian operators and algebraic structures related to
them'', \textit{Functional Analysis and Its Applications}, \textbf{13}:4 (1979), 248-262

\bibitem{Izosimov14}  A.\,M.~Izosimov, ``Generalized argument shift method and complete commutative subalgebras in polynomial Poisson algebras'',  {\tt arXiv:1406.3777 [math.RT] }


\bibitem{IzosimovPriv}  A.\,M.~Izosimov, \text{Private communication}, June 2024



\bibitem{Kozlov24BiLagr} I.\,K.~Kozlov, ``Geometry of bi-Lagrangian Grassmannian'', {\tt arXiv:2409.09855v1 [math.RA] }


\bibitem{Kozlov23JKRealization} I.\,K.~Kozlov,  ``Realization of Jordan-Kronecker invariants by Lie algebras'', {\tt  arXiv:2307.08642v1 [math.DG]}


\bibitem{Kozlov23Shift} I.\,K.~Kozlov,  ``Shifts of semi-invariants and complete commutative subalgebras in polynomial Poisson algebras'',  	 {\tt arXiv:2307.10418 [math.RT]}


\bibitem{Miranda08} C.~Laurent-Gengoux, E.~Miranda, P.~Vanhaecke, ``Action-angle coordinates for integrable systems on Poisson manifolds'', {\tt arXiv:0805.1679 [math.SG] }


\bibitem{Magri78} F.~Magri, ``A simple model of the integrable Hamiltonian equation'', \textit{J. Math. Phys.}, \textbf{19}:5 (1978), 1156-1162


\bibitem{Magri84} F.~Magri, C.~Morosi, ``A geometrical characterization of integrable Hamiltonian systems
through the theory of Poisson-Nijenhuis manifolds''. University of Milano, 1984



\bibitem{ArgShift} A.\,S.~Mishchenko, A.\,T.~Fomenko, ``Euler equations on finite-dimensional Lie groups'',  \textit{Izv. Akad. Nauk SSSR Ser. Mat.}, \textbf{42}:2 (1978), 396-415; \textit{Math. USSR-Izv.}, \textbf{12}:2 (1978), 371-389



\bibitem{Reiman80} A.\,G.~Reiman, M.\,A.~Semenov-Tyan-Shanskii, ``A family of Hamiltonian structures, hierarchy
of hamiltonians, and reduction for first-order matrix differential operators'', \textit{Functional Analysis
and Its Applications}, \textbf{14}:2 (1980), 146–148


\bibitem{Thompson} R.\,C.~Thompson, ``Pencils of complex and real symmetric and skew matrices'', \textit{Linear Algebra Appl.}, \textbf{147} (1991), 323-371 


\bibitem{turiel} F.\,J.~Turiel,  ``Classification locale simultan\'ee de deux formes symplectiques compatibles'', \textit{Manuscripta Mathematica}, \textbf{82}:3--4 (1994), 349-362


\bibitem{Voisin} C.~Voisin, \textit{Hodge Theory and Complex Algebraic Geometry-I},  Cambridge studies in advanced mathematics-76, Cambridge University press, 2002


\end{thebibliography}
\end{document}